\documentclass[11pt, reqno]{article}
\usepackage{amsmath}
\usepackage{amsthm}
\usepackage{amssymb}
\usepackage{latexsym}
\usepackage{color}

\newtheorem{thm}{Theorem}[section]

\newtheorem{lem}{Lemma}[section]
\newtheorem{prop}{Proposition}[section]

\newtheorem{rem}{Remark}[section]

\newtheorem{assumption}{Assumption}[section]
\numberwithin{equation}{section}

\begin{document}
\title{ Nonzero-sum Discrete-time Stochastic Games with Risk-sensitive Ergodic Cost Criterion}

\author{Bivakar Bose\thanks{Department of Mathematics, Indian Institute of Technology Guwahati,
Guwahati-781039, India, email: b.bivakar@iitg.ernet.in}  \;   Chandan Pal\thanks{Department of Mathematics, Indian Institute of Technology Guwahati,
Guwahati-781039, India, email: cpal@iitg.ernet.in}   \;   Somnath Pradhan\thanks{
Department of Mathematics, Indian Institute Science Education and Research,
Bhopal, Bhopal-462066  India, email: somnath@iiserb.ac.in} \; and  Subhamay Saha\thanks{
Department of Mathematics, Indian Institute of Technology,
Guwahati, Guwahati-781039 India, email: saha.subhamay@iitg.ac.in}}
\date{}

\maketitle \baselineskip20pt
\parskip10pt
\parindent.4in
\begin{abstract}
\noindent In this paper we study infinite horizon nonzero-sum stochastic games for controlled discrete-time Markov chains on a Polish state space with risk-sensitive ergodic cost criterion. Under suitable assumptions  we show that the associated ergodic optimality equations admit unique solutions. Finally, the existence of Nash-equilibrium in randomized stationary strategies is established by showing that an appropriate set-valued map has a fixed point.
\end{abstract}

\noindent {\bf Key Words} : Nonzero-sum game, risk-sensitive ergodic cost criterion, optimality equations, Nash equilibrium, set-valued map

\noindent {\bf Mathematics Subject Classification}: 91A15, 91A50
\section{Introduction}
\label{intro}
	We consider infinite horizon ergodic-cost risk-sensitive two-person nonzero-sum stochastic games for discrete time Markov decision processes (MDPs) on a general state space with bounded cost function. We first establish the existence of unique solutions of the corresponding optimality equations. This is used to establish the existence of optimal randomized stationary strategies for a player fixing the strategy of the other player. We then define a suitable topology on the space of randomized stationary strategies. Then under certain separability assumptions we show that a suitably defined set-valued map has a fixed point. This gives the existence of a Nash equilibrium for the nonzero-sum game under consideration. The main step towards this end is to establish the upper semi-continuity of the defined set-valued map.  

In a stochastic control problem, since the cost payable depends on the random evolution of the underlying stochastic process, the same is also random. So the most basic approach is to minimize the expected cost. This is the risk-neutral approach. But the obvious short coming of this approach is that it does not take care of the risk factor. In the risk-sensitive criterion, the expected value of the exponential of the cost is considered. Generally the risk associated with a random quantity is quantified by the standard deviation. Hence risk-sensitive criterion provides significantly better  protection from the risk since it captures the effects of the higher order moments of the cost as well, see \cite{whittle1981risk} for more details. The analysis of risk-sensitive control is technically more involved because of the exponential nature of the cost; accumulated costs are multiplicative over time as opposed to additive in the risk-neutral case. Risk-sensitive control problems also has deep connections to robust control theory, control theory literature which deals with model uncertainty, see \cite{whittle1981risk} and references therein. Risk-sensitive criterion also arises naturally in portfolio optimization problems in mathematical finance, see \cite{Sheu05} and associated references. Nonzero-sum stochastic games arise naturally in strategic multi-agent decision-making system like socioeconomic systems \cite{Jackson08} , network security \cite{basar11}, routing and scheduling \cite{Roughgarden_2016}  and so on.

Following the seminal work of of Howard and Matheson \cite{Matheson72} there has been a lot of work on risk-sensitive control problems. For an up to date survey on ergodic risk-sensitive control problems with underlying stochastic process being discrete-time Markov chains, see \cite{Borkar23}. In the multi-controller set up, discrete-time risk-sensitive zero-sum games on countable state space have been studied in  \cite{Ghosh14,ghosh2023discrete}. In \cite{Rieder17} the authors consider discrete-time zero-sum risk-sensitive stochastic games on Borel state space for both discounted and ergodic cost criterion. Their analysis involves transforming the original risk-sensitive problem to a  risk-neutral problem. Nonzero-sum risk-sensitive stochastic games with both discounted and ergodic cost criterion for countable state space and bounded cost has been studied in \cite{basu2018nonzero, wei2019risk}. In \cite{wei2021risksensitivenonzero}  discrete-time nonzero-sum stochastic games under the risk-sensitive ergodic cost criterion with countable state space and unbounded cost function has been investigated. To the best of our knowledge, this is the first paper which considers risk-sensitive nonzero-sum games on a general state space. The analysis of nonzero-sum games on general state space is significantly more involved as compared to the case of countable state space. One of the reasons for this is that in case of countable state space the topology on the space of randomized stationary controls is fairly simple to work with. In case of general state space this becomes a substantial challenge. In the risk-neutral setup itself the analysis becomes substantially involved and requires additional assumptions. In the risk-neutral set up discrete-time nonzero-sum ergodic cost games with general state space has been studied in \cite{ghosh1998stochastic, Altman02, Chen_discrete19}. In \cite{ghosh1998stochastic}, the authors prove the existence of a Nash-equilibrium under an additive reward additive transition(ARAT) assumption. In \cite{Altman02}, the authors establish the existence of an $\epsilon$-Nash equilibrium. In \cite{Chen_discrete19}, the authors assume that the transition law is a combination of finitely many probability measures. In this work we also assume ARAT condition.  

In this paper we first consider ergodic optimality equations which correspond to control problems with the strategy of one player fixed. We show the existence of unique solutions to these equations. We follow a span contraction approach. For this analysis we make certain ergodicity assumption along with a few extra assumptions on the state transition kernel. Analogous assumptions appear in \cite{Stettner99} where the authors consider risk-sensitive control problems with ergodic cost criterion. In order to establish the existence of a Nash equilibrium we first define an appropriate set valued map. In order to  to show the existence of a Nash equilibrium for the nonzero-sum game we use Fan's fixed point theorem \cite{Fan1952fixed}. This involves showing that the defined set valued map is upper semi-continuous under the appropriate topology.  

The rest of the paper is organized as follows. Section 2 introduces the game model, some preliminaries and  notations. In section 3 we show the existence of a unique solution to the optimality equation using a certain span-norm contraction. Existence of the Nash-equilibrium  has been shown in section 4.
\section{The Game Model}
In this section, we present the discrete-time nonzero-sum stochastic game model and introduce the notations utilized throughout the paper.
The following elements are needed to describe the discrete-time nonzero-sum stochastic game:
\begin{align*}\label{model}
	\{X,A,B,P,c_1, c_2\},
\end{align*}	where each component is described below.

\begin{itemize}
	\item $X$ is the state space assumed to be Polish space endowed with the Borel $\sigma$-algebra $\mathcal{X}$. 
	
	\item $A$ and $B$ are action spaces of player $1$ and $2$ respectively, assumed to be compact metric spaces. Let $\mathcal{A}$ and $\mathcal{B}$ denote the Borel $\sigma$-algebras on $A$ and $B$ respectively.

	\item $P$ is the transition kernel from $X\times A\times B \to \mathcal{P}(X)$, where for any metric space $D$, $\mathcal{P}(D)$ denotes the space of all probability measures on $D$ with the topology of weak convergence.
	
	\item $c_i:X\times A\times B \to \mathbb{R}$, $i=1,2,$ is one-stage cost function for player $i$, assumed to be bounded and continuous on $A\times B$. Since, the cost is bounded, without loss of generality let, $0\leq c_i\leq \bar{c}$.
\end{itemize} 

At each stage(time) the players observe the current state $x \in X$ of the system and then player $1$ and $2$ independently choose actions $a \in A$ and $ b\in B$ respectively. As a consequence two things happen:

\begin{itemize}
	\item[(i)] player $i, i=1,2$, pays an immediate cost $c_i(x,a,b).$
	
	\item[(ii)] the system moves to a new state $y$ with the distribution $P(\cdot|x,a,b).$
\end{itemize}

The whole process then repeats from the new state $y$. The cost accumulates throughout the course of the game. The planning horizon is taken to be infinite and each player wants to minimize his/her infinite-horizon risk-sensitive cost with respect to some cost	criterion, which is in our case defined by \eqref{cost criterion} below.

At each stage the players choose their actions independently on the basis of the available information. The information available for decision making at time $t \in \mathbb{N}_0:=\{0,1,2, \ldots\}$ is given by the history of the process up to that time
\begin{equation*}
	h_t:=\left(x_0, a_0, b_0, x_1, a_1, b_1, \ldots, a_{t-1}, b_{t-1}, x_t\right) \in H_t,
\end{equation*}
where $H_0=X, H_t=H_{t-1} \times(A \times B \times X), \ldots, H_{\infty}=(X \times A \times B)^{\infty}$ are the history spaces. The history spaces are endowed with the corresponding Borel $\sigma-$algebra. A strategy for player 1 is a sequence $\pi^1=\left\{\pi_t^1\right\}_{t \in \mathbb{N}_0}$ of stochastic kernels $\pi_t^1: H_t \rightarrow \mathcal{P}(A)$. The set of all strategies for player 1 is denoted by $\Pi_1$. A strategy $\pi^1 \in \Pi_1$ is called a Markov strategy if
\begin{equation*}
	\pi_t^1\left(h_{t-1}, a, b, x\right)(\cdot)=\pi_t^1\left(h_{t-1}^{\prime}, a^{\prime}, b^{\prime}, x\right)(\cdot)
\end{equation*}
for all $h_{t-1}, h_{t-1}^{\prime} \in H_{t-1}, a, a^{\prime} \in A, b,b^{\prime}  \in B, x \in X, t \in \mathbb{N}_0$. Thus a Markov strategy for player 1 can be identified with a sequence of measurable maps $\left\{\varPhi_t^1\right\}, \varPhi_t^1: X \rightarrow \mathcal{P}(A)$. A Markov strategy $\left\{\Phi_t^1\right\}$ is called a stationary strategy if $\varPhi_t^1=\Phi: X \rightarrow \mathcal{P}(A)$ for all $t$. Let $M_1$ and $S_1$ denote the set of Markov and stationary strategies for player 1 , respectively. The strategies for player 2 are defined similarly. Let $\Pi_2, M_2,$ and $S_2$ denote the set of arbitrary, Markov and stationary strategies for player 2, respectively.

Given an initial distribution $\tilde{\pi}_0 \in \mathcal{P}(X)$ and a pair of strategies $(\pi^1, \pi^2) \in \Pi_1 \times \Pi_2$, the corresponding state and action processes $\left\{X_t\right\},\left\{A_t\right\},\left\{B_t\right\}$ are stochastic processes defined on the canonical space $(H_{\infty}, \mathfrak{B}(H_{\infty}), P_{\tilde{\pi}_0}^{\pi^1, \pi^2})$(where $\mathfrak{B}(H_{\infty})$ is the Borel $\sigma$-field on $H_{\infty}$) via the projections $X_t\left(h_{\infty}\right)=x_t, A_t\left(h_{\infty}\right)=a_t, B_t\left(h_{\infty}\right)=b_t$, where $P_{\tilde{\pi}_0}^{\pi^1, \pi^2}$ is uniquely determined by $\pi^1, \pi^2$ and $\tilde{\pi}_0$ by Ionescu Tulcea's theorem \cite[Proposition 7.28]{Bertsekas1996stochastic}.  When $\tilde{\pi}_0=\delta_{x_0}$ (the dirac measure at $x_0$),$~ x_0\in X$, we simply write this probability measure as $P_{x_0}^{\pi_1, \pi_2}$. For $h_{t}\in H_{t}, a \in A, b\in B$, we have
\begin{align*}
	&P_{x_0}^{\pi^1, \pi^2}(X_0=x_0)=1,\\
	&P_{x_0}^{\pi^1, \pi^2}(X_{t+1}\in E|h_{t},A_t=a,B_t=b)=P(E|x_t,a,b) ~\forall E\in \mathcal{X},\\
	&P_{x_0}^{\pi^1, \pi^2}(A_t\in F, B_t\in G|h_{t})=\pi_t^1(h_{t})(F)\pi_t^2(h_{t})(G) ~\forall F\in \mathcal{A}, ~\forall G\in \mathcal{B}.
\end{align*} 

Let the corresponding expectation operator be denoted by $E_{{\tilde{\pi}}_0}^{\pi^1, \pi^2}(E_{x_0}^{\pi^1, \pi^2})$ with respect to the probability measure $P_{{\tilde{\pi}}_0}^{\pi^1, \pi^2}(P_{x_0}^{\pi^1, \pi^2})$. Now from \cite{Hernandez2012adaptive} we know that under any $(\Phi\times \Psi) \in S_1\times S_2$, the corresponding stochastic process $X_t$ is a Markov process.

\noindent {\bf Ergodic cost criterion:} We now define the risk-sensitive ergodic cost criterion for nonzero-sum discrete-time game. Let $(X_t,A_t,B_t)$ be the corresponding process with $X_0=x\in X$ and $\theta>0$ be the risk-sensitive parameter. For a pair of strategies $(\pi^1,\pi^2)\in (\Pi_1\times \Pi_2)$, the risk-sensitive ergodic cost criterion for player $i=1,2$ is given by	
\begin{equation}\label{cost criterion}
	J^{\pi^1,\pi^2}_{i}(x)=\limsup_{n\to \infty}\frac{1}{n}lnE^{\pi^1,\pi^2}_x \bigg[e^{{\theta}{\sum_{t=0}^{n-1}c_i(x_t,a_t,b_t)}}\bigg]
\end{equation}
Since the risk-sensitive parameter remains the same throughout, we assume without loss of generality that $\theta = 1$. Note that, $J^{\pi^1,\pi^2}_{i}$ for $i=1,2$ are bounded as our cost functions are bounded.

\noindent {\bf Nash equilibrium}: A pair of strategies $(\pi^{*1},\pi^{*2})\in \Pi_1\times\Pi_2$ is called a Nash equilibrium(for the ergodic cost criterion) if
\begin{equation*}
	J^{\pi^{*1},\pi^{*2}}_{1}(x)\leq J^{\pi^1,\pi^{*2}}_{1}(x)~ \text{for all } \pi^1\in \Pi_1 ~\text{and}~ x\in X
\end{equation*}and
\vspace{1mm}
\begin{equation*}
	J^{\pi^{*1},\pi^{*2}_2}(x)\leq J^{\pi^{*1},\pi^2}_{2}(x)~ \text{for all } \pi^2\in \Pi_2~ \text{and}~ x\in X.
\end{equation*}  Our primary goal is to establish the existence of a Nash equilibrium in stationary strategies.

For $i=1,2$, let us define transition measures $\tilde{P}_i$ from $X\times A\times B \to\mathcal{P}(X)$ by
\begin{equation}\label{Pi tilde}
	\tilde{P}_i(dy|x,a,b)=e^{c_i(x,a,b)}P(dy|x,a,b) .
\end{equation} 
Moreover, define for  $i=1,2 ,~\varphi \in \mathcal{P}(A)$ and  $\psi \in \mathcal{P}(B)$
\begin{equation*}
	P(dy|x,\varphi,\psi):=\int_B \int_A P(dy|x,a,b)\varphi (da) \psi (db)
\end{equation*}
and
\begin{equation*}
	\tilde{P}_i(dy|x,\varphi,\psi):=\int_B \int_A \tilde{P}_i(dy|x,a,b)\varphi (da) \psi (db).
\end{equation*}
Obviously $\tilde{P}_i$ for $i=1,2$ is in general is not a probability measure. The normalizing constant for $x\in X, \varphi \in \mathcal{P}(A),$ and $\psi \in \mathcal{P}(B)$ is given by
\begin{equation}\label{c}
	\tilde{c}_i(x,\varphi,\psi)=\int_X 	\tilde{P}_i(dy|x,\varphi,\psi)=\int_B\int_A e^{c_i(x,a,b)}\varphi (da) \psi (db).
\end{equation} Since $0\leq c_i\leq \bar{c}$, the function $\tilde{c}_i$ is also bounded. More precisely $1\leq \tilde{c}_i(x,\varphi,\psi)\leq e^{\bar{c}} $ for  $i=1,2$ and for each $x\in X, \varphi\in \mathcal{P}(A)$ and $\psi\in \mathcal{P}(B)$. 
Thus for $i=1,2$ \begin{equation}\label{P2}
	\hat{P}_i(\cdot|x,\varphi,\psi):=\frac{\tilde{P}_i(\cdot|x,\varphi,\psi)}{\tilde{c}_i(x,\varphi,\psi)}
\end{equation}	defines a probability transition kernel and we also use the notation $\hat{c}_i(x,\varphi,\psi)$:= $\ln \tilde{c}_i(x,\varphi,\psi)$ for $i=1,2$. 

We will use the above transformations, to convert our optimality equations \eqref{main equation1} and \eqref{main equation2} into a well-known equation. This process is beneficial as it helps us to prove the existence of the unique solution to the optimality equations as we will see in the next section.

Define $\mathbb{B}(X)$, the space of all real valued bounded measurable functions on $X$ endowed with the supremum norm $\|.\|$. For a fixed $ \varphi \in \mathcal{P}(A),$  and  $v\in \mathbb{B}(X)$  define the operator:
\begin{equation}\label{operator1}
	Tv(x)=\inf_{\psi \in \mathcal{P}(B)}\bigg[\hat{c}_2(x,\varphi,\psi)+ln \int_X e^{v(y)}\hat{P}_2(dy|x,\varphi,\psi)\bigg].
\end{equation}
Due to the dual representation of the exponential certainty equivalent \cite[Lemma 3.3]{Hernandez1996risk} it is possible to write \eqref{operator1} as
\begin{equation}\label{entropy operator}
	Tv(x)=\inf_{\psi}\sup_{\mu}\bigg[\hat{c}_2(x,\varphi,\psi)+ \int_X v(y)\mu(dy)-I(\mu,\hat{P}_2(\cdot |x,\varphi,\psi)\bigg],
\end{equation} where the supremum is over all probability measures $\mu \in \mathcal{P}(X)$ and $I(p,q)$ is the relative entropy of the two probability measures $p,q$ which is defined by 
\begin{equation*}
	I(p,q):= \int_X ln\frac{dp}{dq}p(dx)
\end{equation*} 
when $p<< q$ and $+\infty$ otherwise.	Note that the supremum in \eqref{entropy operator} is attained at the probability measure given by 
\begin{equation}\label{psi}
	\mu^{\varphi}_{x \psi v}(E):=\frac{\int_{E}e^{v(y)}\hat{P}_2(dy |x,\varphi,\psi)}{\int_{X}e^{v(y)}\hat{P}_2(dy |x,\varphi,\psi)}
\end{equation} for measurable set $E \subset X$. Obviously 	$\mu^{\varphi}_{x \psi v}$ can be interpreted as a transition kernel.

Let $v\in \mathbb{B}(X)$. The span semi-norm of $v$ is defined as:
\begin{equation}\label{span norm}	
	\|v\|_{sp}=\sup_{x \in X} v(x)-\inf_{x\in X} v(x).
\end{equation}

In the last part of this section we make a couple of assumptions that will be in force throughout the rest of the paper. First we will consider the following continuity assumption which is quite standard in literature, see \cite{Rieder17,Stettner99,ghosh1998stochastic,Parthasarathy1982existence} for instance. This will allow us to show that the map $T$ defined in \eqref{operator1} is a contraction in the span semi-norm.

\begin{assumption}\label{assumption1}
	
	For a fixed $x\in X$ the transition measure $P$ is strongly continuous in $(a,b)$, i.e. for all bounded and measurable $v: X \to \mathbb{R}$ we have that  $(a,b) \mapsto \int_X v(y)P(dy|x,a,b)$ is continuous in $(a,b)$.
\end{assumption}
It follows  from \eqref{Pi tilde} that $\tilde{P}_i(dy|x,a,b)$  is also strongly continuous in $(a,b)$ for $i=1,2$ which lead us the  following remark.

\begin{rem}\label{remark1}
	Due to the fact that we consider for any metric space $D$, the space $\mathcal{P}(D)$ is  endowed with the topology of weak convergence, from  Assumption \ref{assumption1} it follows immediately that for all bounded and measurable functions $v:X \mapsto \mathbb{R}$ and a fixed $x\in X$ the map $(\varphi, \psi)\mapsto \int_{X}v(y)\hat{P}_i(dy|x, \varphi, \psi),$
	$i=1,2, $ is continuous in $(\varphi,\psi) \in \mathcal{P}(A)\times\mathcal{P}(B)$.
\end{rem} 	 	
Next we have the following ergodicity assumption. 
\begin{assumption}\label{assumption2}
	
	\item[(i)] There exists a real number $0<\delta<1$ such that
	\begin{equation*}
		\sup \left\|{P}(\cdot | x, \varphi, \psi)-{P}\left(\cdot \mid x^{\prime}, \varphi^{\prime}, \psi^{\prime}\right)\right\|_{\mathrm{TV}} \leq 2 \delta,
	\end{equation*}
	where the supremum is over all $x, x^{\prime} \in X, \varphi, \varphi^{\prime} \in \mathcal{P}(A), \psi, \psi^{\prime} \in \mathcal{P}(B)$ and $\|\cdot\|_{\mathrm{TV}}$ denotes the total variation norm.
	\item[(ii)] For $x \in X, a \in A, b \in B,~ P(\mathcal{O} | x, a, b)>0$ for any open set $\mathcal{O} \subset X$.
\end{assumption}

	\section{Solution to the optimality equations}

In this section, we demonstrate that the operator $T$ defined in \eqref{operator1} is a contraction. The fixed point of  $T$  corresponds to the solution of the optimality equation for player 1. In the latter part of this section, we define another operator $U$ corresponding to player 2 and establish results analogous to those obtained for $T$.

\begin{prop}\label{theorem 5.4}
	Under Assumptions \ref{assumption1} and \ref{assumption2} the operator  $T$ maps $\mathbb{B}(X)$ to $\mathbb{B}(X)$ and for each $M>0$, there exists a positive constant $\alpha(M)<1$ such that for   all $v_1, v_2 \in \mathbb{B}_M(X)$
	\begin{equation*}
		\left\|T v_1-T v_2\right\|_{sp} \leq \alpha(M)\left\|v_1-v_2\right\|_{sp},
	\end{equation*}
	where $\mathbb{B}_M(X)=\{v\in \mathbb{B}(X):\|v\|_{sp}\leq M \}$.
\end{prop}
\begin{proof} From the definition of $T$ it follows that $T$ transforms $\mathbb{B}(X)$ into itself and the infimum is attained.
	For given functions $v_1,v_2\in\mathbb{B}(X) $ and $x_1,x_2\in X$, let $\Psi_1,\Psi_2 \in S_2$ be such that for $i=1,2$
	\begin{equation*}
		T v_i(x_i)=\sup_\mu\big[\hat{c}_2(x_i,\varphi,\Psi_{i}(x_i))+ \int_X v_i(y)\mu(dy)-I(\mu,\hat{P}_2(\cdot |x_i,\varphi,\Psi_{i}^*(x_i))\big].	
	\end{equation*}
	Then we obtain that	
	\begin{equation*}
		\begin{aligned}
			& (T v_1)(x_2)-(T v_2)(x_2)-((T v_1)(x_1)-(T v_2)(x_1)) \\
			& \leq \sup _\mu\left\{\hat{c}_2\left(x_2, \varphi, \Psi_2(x_2)\right)+\int_X v_1\left(y\right) \mu(d y )- I\left(\mu, \hat{P_2}\left(\cdot | x_2, \varphi, \Psi_2(x_2)\right)\right)\right\} \\
			&-  \sup _\mu\left\{\hat{c}_2\left(x_2, \varphi, \Psi_2(x_2)\right)+\int_X v_2(y) \mu(d y )- I\left(\mu, \hat{P_2}\left(\cdot | x_2, \varphi, \Psi_2(x_2)\right)\right)\right\}\\
			& -\sup _\mu\left\{\hat{c}_2\left(x_1, \varphi, \Psi_1(x_1)\right)+\int_X v_1\left(y\right) \mu\left(d y\right )-I\left(\mu, \hat{P_2}\left(\cdot | x_1, \varphi, \Psi_1(x_1)\right)\right)\right\} \\
			& +\sup _\mu\left\{\hat{c}_2\left(x_1, \varphi, \Psi_1(x_1)\right)+\int_X v_2\left(y\right) \mu\left(d y \right)- I\left(\mu, \hat{P_2}\left(\cdot | x_1, \varphi, \Psi_1(x_1)\right)\right)\right\}\\
			& \leq \hat{c}_2\left(x_2, \varphi, \Psi_2(x_2)\right)+\int_X v_1\left(y\right) \mu^{\varphi}_{x_2\Psi_2v_1}(d y )- I\left( \mu^{\varphi}_{x_2\Psi_2v_1}, \hat{P_2}\left(\cdot | x_2, \varphi, \Psi_2(x_2)\right)\right) \\
			&-  \hat{c}_2\left(x_2, \varphi, \Psi_2(x_2)\right)-\int_X v_2(y) \mu^{\varphi}_{x_2\Psi_2v_1}(d y )+ I\left( \mu^{\varphi}_{x_2\Psi_2v_1}, \hat{P_2}\left(\cdot | x_2, \varphi, \Psi_2(x_2)\right)\right)\\
			&- \hat{c}_2\left(x_1, \varphi, \Psi_1(x_1)\right)-\int_X v_1\left(y\right) \mu^{\varphi}_{x_1\Psi_1v_2}(d y )+I\left(\mu^{\varphi}_{x_1\Psi_1v_2}, \hat{P_2}\left(\cdot | x_1, \varphi, \Psi_1(x_1)\right)\right) \\
			&+ \hat{c}_2\left(x_1, \varphi, \Psi_1(x_1)\right)+\int_X v_2\left(y\right) \mu^{\varphi}_{x_1\Psi_1v_2}(d y )-I\left(\mu^{\varphi}_{x_1\Psi_1v_2}, \hat{P_2}\left(\cdot | x_1, \varphi, \Psi_1(x_1)\right)\right) \\	&=\int_{\Delta}\left(v_1\left(y\right)-v_2\left(y\right)\right) \left(\mu^{\varphi}_{x_2\Psi_2v_1}-\mu^{\varphi}_{x_1\Psi_1v_2} \right)(d y )+
			\int_{\Delta^c}\left(v_1\left(y\right)-v_2\left(y\right)\right) \left(\mu^{\varphi}_{x_2\Psi_2v_1}-\mu^{\varphi}_{x_1\Psi_1v_2} \right)(d y )\\
			& \leq \sup_{y\in X} \left(v_1\left(y\right)-v_2\left(y\right)\right) \left(\mu^{\varphi}_{x_2\Psi_2v_1}-\mu^{\varphi}_{x_1\Psi_1v_2} \right)(\Delta)+ 
			\inf_{y\in X}\left(v_1\left(y\right)-v_2\left(y\right)\right) \left(\mu^{\varphi}_{x_2\Psi_2v_1}-\mu^{\varphi}_{x_1\Psi_1v_2} \right)({\Delta}^c)\\
			&=\|v_1-v_2\|_{sp}(\mu^{\varphi}_{x_2\Psi_2v_1}-\mu^{\varphi}_{x_1\Psi_1v_2})(\Delta),
		\end{aligned}
	\end{equation*}
	where the set $\Delta$ comes from the Hahn-Jordan decomposition of $\mu^{\varphi}_{x_2\Psi_2v_1}-\mu^{\varphi}_{x_1\Psi_1v_2}$ and $\Delta^c$ denotes the complement of $\Delta$. Now, taking supremum over $x_1,x_2\in X$ in the above set-up we have
	\begin{equation*}
		\left\|T v_1-T v_2\right\|_{sp} \leq \|v_1-v_2\|_{sp}\sup_{E\in \mathcal{X}}\sup_{x_1,x_2\in X}\sup_{\Psi_1 ,\Psi_2 \in S_2}(\mu^{\varphi}_{x_2\Psi_2v_1}-\mu^{\varphi}_{x_1\Psi_1v_2})(E).
	\end{equation*} 	
	We claim that 
	\begin{equation}\label{1.12}
		\sup_{v_1,v_2;\|v_1\|_{sp},\|v_2\|_{sp}\leq M}\sup_{E\in \mathcal{X}}\sup_{x_1,x_2\in X}\sup_{\Psi_1 ,\Psi_2 \in S_2}(\mu^{\varphi}_{x_2\Psi_2v_1}-\mu^{\varphi}_{x_1\Psi_1v_2})(E)=\alpha(M)<1.
	\end{equation}
	Suppose \eqref{1.12} does not hold. Then there exists sequences $\{v_{1n}\},\{v_{2n}\}$ with $\|v_{1n}\|_{sp}\leq M,\|v_{2n}\|_{sp}\leq M$, $\{E_n\}, E_n\in \mathcal{X},\{x_{1n}\},\{x_{2n}\}$ and $\{\Psi_{1n}\},\{\Psi_{2n}\}$ such that 
	\begin{equation*}
		(\mu^{\varphi}_{x_{2n}\Psi_{2n}v_{1n}}-\mu^{\varphi}_{x_{1n}\Psi_{1n}v_{2n}})(E_n)\rightarrow 1 ~~~\text{as}~~~ n\to \infty.
	\end{equation*}
	As $\mu^{\varphi}_{x_{2n}\Psi_{2n}v_{1n}}$ and $\mu^{\varphi}_{x_{1n}\Psi_{1n}v_{2n}}$ are probability measures therefore
	\begin{equation*}
		\mu^{\varphi}_{x_{2n}\Psi_{2n}v_{1n}}(E_n)\to 1 ~~~\text{as}~~~ n\to \infty,
	\end{equation*}
	and
	\begin{equation*}
		\mu^{\varphi}_{x_{1n}\Psi_{1n}v_{2n}}(E_n)\to 0 ~~~\text{as}~~~ n\to \infty. 
	\end{equation*}	
	Since for each $x\in X, \psi \in \mathcal{P}(B)$ and $v\in \mathbb{B}(X)$ from \eqref{psi} we get
	\begin{equation*}
		e^{-\|v\|_{sp}} \hat{P_2}\left(E | x, \varphi, \psi\right)\leq 	\mu^{\varphi}_{x \psi v}(E),
	\end{equation*}
	we have
	\begin{equation*}
		\hat{P_2}\left(E_n^c |x_{2n}, \varphi, \Psi_{2n}(x_{2n})\right)\to 0
		~~~\text{as}~~~ n\to \infty,	\end{equation*}
	and
	\begin{equation*}
		\hat{P_2}\left(E_n | x_{1n}, \varphi, \Psi_{1n}(x_{1n})\right)\to 0
		~~~\text{as}~~~ n\to \infty.
	\end{equation*}
	Consequently using 	\eqref{P2}  direct calculations imply 
	
	\begin{equation*}
		P\left(E_n^c |x_{2n}, \varphi, \Psi_{2n}(x_{2n})\right)\to 0
		~~~\text{as}~~~ n\to \infty,\end{equation*}
	and
	\begin{equation*}
		P\left(E_n | x_{1n}, \varphi, \Psi_{1n}(x_{1n})\right)\to 0
		~~~\text{as}~~~ n\to \infty.
	\end{equation*} Hence
	\begin{equation}\label{1.13}
		\lim_{n\to \infty}\bigg(P\left(E_n|x_{2n}, \varphi, \Psi_{2n}(x_{2n})\right)-P\left(E_n | x_{1n}, \varphi, \Psi_{1n}(x_{1n})\right)\bigg)=1
	\end{equation}
	But from   Assumption \ref{assumption2}$(i)$ we get,  
	$\forall x_1,x_2\in X, \forall \psi,\psi^{\prime}\in \mathcal{P}(B), \forall E\in \mathcal{X}$ 
	\begin{equation*}\label{1.14}
		P(E | x, \varphi, \psi)-P(E| x^{\prime}, \varphi, \psi^{\prime}) \leq  \delta,
	\end{equation*} which contradicts \eqref{1.13}. Hence  \eqref{1.12} holds and therefore the theorem also holds true. 
\end{proof}	
We will now make additional assumptions to show that $T$ is a global contraction in $\mathbb{B}_L(X).$ 

\begin{assumption}\label{assumption1.4}
	There exists $\lambda \in \mathcal{P}(X)$ such that $P(\cdot|x,a,b)<<\lambda$ for all $x\in X, a\in A,$ and $b\in B.$ Also let $h:X\times A\times B\times X \to \mathbb{R}$ be the Radon-Nikodym derivative of $P(\cdot|x,a,b)$ with respect to $ \lambda$.
\end{assumption}

\begin{assumption}\label{assumption1.5}
	\begin{equation*}\sup_{x,x'\in X}\sup_{y\in X}\sup_{a\in A}\sup_{b\in B} \frac{h(x,a,b,y)}{h(x',a,b,y)}=\kappa<\infty.\end{equation*}
\end{assumption}

\begin{lem}\label{global contraction}
	Under the Assumptions \ref{assumption1}, \ref{assumption1.4} and  \ref{assumption1.5}  the operator $T$ transforms $\mathbb{B}(X)$ into $\mathbb{B}_L(X)$, where  $L=ln\kappa+3\bar{c} $. Furthermore, $T$ is a global contraction in $\mathbb{B}_L(X).$
\end{lem}
\begin{proof}  Notice that for a $v\in \mathbb{B}(X)$ we have 
	\begin{equation}\label{global contraction equation}
		Tv(x)-Tv(x')\leq \sup_{ \psi \in \mathcal{P}(B)}\Bigg[\hat{c}_2(x, \varphi, \psi)-\hat{c}_2(x', \varphi, \psi)+ln \frac{\int_X e^{v(y)}\hat{P}_2(dy|x, \varphi, \psi)}{\int_X e^{v(y)}\hat{P}_2(dy|x', \varphi, \psi)} \Bigg].
	\end{equation}
	
	Now we get
	\begin{equation*}
		\begin{aligned}
			&ln\frac{\int_X e^{v(y)}\hat{P}_2(dy|x, \varphi, \psi)}{\int_X e^{v(y)}\hat{P}_2(dy|x', \varphi, \psi)}\\
			&=ln\frac{\int_X e^{v(y)}\int_{B}\int_{A} e^{c_{2}(x,a,b)}{\frac{h(x,a,b,y)}{h(x',a,b,y)}}h(x',a,b,y)\varphi(da)\psi(db) \lambda(dy)}{\int_X e^{v(y)}\int_{B}\int_{A} e^{c_{2}(x',a,b)}h(x',a,b,y)\varphi(da)\psi(db) \lambda(dy)}. \frac{\tilde{c}_2(x', \varphi, \psi)}{\tilde{c}_2(x, \varphi, \psi)}\\
			&\leq ln\kappa+ln\frac{\tilde{c}_2(x', \varphi, \psi)}{\tilde{c}_2(x, \varphi, \psi)}+ln\frac{\int_X e^{v(y)}\int_{B}\int_{A} e^{c_{2}(x,a,b)}h(x',a,b,y)\varphi(da)\psi(db) \lambda(dy)}{\int_X e^{v(y)}\int_{B}\int_{A} e^{c_{2}(x',a,b)}h(x',a,b,y)\varphi(da)\psi(db) \lambda(dy)}\\
			&\leq  ln\kappa+lne^{\bar{c}}+ln\frac{e^{\bar{c}}\int_X e^{v(y)}\int_{B}\int_{A} h(x',a,b,y)\varphi(da)\psi(db) \lambda(dy)}{\int_X e^{v(y)}\int_{B}\int_{A} h(x',a,b,y)\varphi(da)\psi(db) \lambda(dy)}\\
			=&ln\kappa+2\bar{c} .
		\end{aligned}
	\end{equation*}
	
	In the above expression the first equality follows from Assumption \ref{assumption1.4}, first inequality follows from Assumption \ref{assumption1.5} and the last inequality  follows from \eqref{c} and the fact that $0\leq c_2\leq \bar{c}$. So, from \eqref{global contraction equation} we have 
	\begin{equation}\label{line263}
		Tv(x)-Tv(x')\leq ln\kappa+3\bar{c}.
	\end{equation}
	
	From \eqref{line263} it follows that $\|Tv\|_{sp}\leq L$. Hence $T$ is a global contraction in the span norm on $ \mathbb{B}_L(X).$	
\end{proof}	
Before proceeding to the main theorem of this section, we briefly outline some main points. Suppose player 2 announces that he/she is going to employ a strategy $\Psi \in S_2$. In such a scenario, player 1 attempts to minimize
\begin{equation*}
	J^{\pi^1,\Psi}_{1}(x)=\limsup_{n\to \infty}\frac{1}{n}lnE^{\pi^1,\Psi}_x \bigg[e^{\sum_{t=0}^{n-1}c_i(x_t,a_t,b_t)}\bigg]
\end{equation*}
over $\pi^1 \in \Pi_1$. Thus for player 1 it is a discrete-time Markov decision problem with risk sensitive ergodic cost. Player 2 go through equivalent situations when player 1 announces his strategy to be $\Phi \in S_1$. 
%
This leads us to the following theorem.

\begin{thm}\label{theoremmain1}
	Suppose Assumptions \ref{assumption1}, \ref{assumption2}, \ref{assumption1.4} and \ref{assumption1.5} are satisfied.Then for $\Phi \in S_1$, there exists a unique solution pair  $(\rho_2^*,v_2^*)\in \mathbb{R_+}\times  \mathbb{B}_L(X) $ 	with $v_2^*(x_0)=0$, satisfying
	\begin{equation}\label{main equation1}
		e^{v(x)+\rho}= {\inf_{ \psi \in \mathcal{P}(B)} \int_B\int_A e^{c_2(x,a,b)}\int_X e^{v(y)}P(dy|x,a,b)\Phi(x) (da)  \psi (db)}.
	\end{equation}
	In addition, a strategy $\Psi^*\in S_2$ is an optimal strategy of player 2 given player 1 chooses $\Phi$ if and only if \eqref{main equation1} attains point-wise minimum at $\Psi^*$. Moreover, 
	\begin{equation}\label{rho representation1}
		\rho^*_2=\inf_{\pi^2\in \Pi_2}\limsup_{n\to \infty}\frac{1}{n}lnE^{\Phi,\pi^2}_x \bigg[e^{\sum_{t=0}^{n-1}c_2(x_t,a_t,b_t)}\bigg].
	\end{equation} 
	\hspace{3cm}$\bigg(:=\rho^{*\Phi}_2=\inf_{\pi^2\in \Pi_2}J^{\Phi,\pi^2}_2\bigg) $
\end{thm}

\begin{proof} Notice that \eqref{main equation1} can be rewritten as 
	\begin{equation}\label{transformed equation}
		v(x)+\rho=\inf_{ \psi \in \mathcal{P}(B)}\bigg[\hat{c}_2(x,\Phi(x), \psi)+ln \int_X e^{v(y)}\hat{P}_2(dy|x,\Phi(x), \psi)\bigg].
	\end{equation}

	By Lemma \ref{global contraction},  $T$ is a global contraction in the span norm in $ \mathbb{B}_L(X)$, so that it has a fixed point $\hat{v}_2$ in $ \mathbb{B}_L(X)$ and $\hat{v}_2$ (which is unique up to an additive constant) and the constant $\rho_2^*=T\hat{v}_2-\hat{v}_2$ are solutions to \eqref{transformed equation} and consequently to \eqref{main equation1}. 
	
	Let $v^*_2(x)=\hat{v}_2(x)-\hat{v}_2(x_0)$. Then $v^*_2(x_0)=0$
	and it can be easily seen that $(\rho_2^*,v^*_2)$ satisfies \eqref{transformed equation}. Since, $v^{*}_2(x)=\hat{v}_2(x)-\hat{v}_2(x_0)$
	and $\|\hat{v}_2\|_{sp}\leq L$, so $\|v^{*}_2\|_{sp}\leq L$ as well.
	
	Let $(\rho',v')$ be another solution of \eqref{main equation1} i.e., it satisfies $\rho'+v'=Tv'(x)$ with $v'(x_0)=0$. Then clearly $v'$ is also a span fixed point of $T$. Hence $v^*_2(x)-v'(x)=constant.$ Since $v^*_2(x_0)-v'(x_0)=0$, it follows that $v^*_2\equiv v'$. It then easily follows that $\rho^*_2=\rho'$.
	
	The proof of the remaining part is analogous to the proof  in \cite[Theorem 2.1]{Hernandez1996risk} which has been done for countable state space but can easily be extended to our general state space case.
\end{proof}	
\vspace{.5cm}

For a fixed $  \psi \in \mathcal{P}(B),$  and  $v\in \mathbb{B}(X)$  define the operator:
\begin{equation}\label{operator2}
	Uv(x)=\inf_{ \varphi \in \mathcal{P}(A)}\bigg[\hat{c}_1(x, \varphi, \psi)+ln \int_X e^{v(y)}\hat{P}_1(dy|x, \varphi, \psi)\bigg].
\end{equation}

By similar arguments we can also show that $U$ is a global contraction in the span norm in $ \mathbb{B}_L(X)$ and the following theorem holds true.

\begin{thm}\label{theoremmain2}
	Suppose  Assumptions \ref{assumption1}, \ref{assumption2},  \ref{assumption1.4} and \ref{assumption1.5} are satisfied. Then	for $\Psi \in S_2$, there exists a unique solution pair  $(\rho^{*}_1,v_1^{*})\in \mathbb{R_+}\times  \mathbb{B}_L(X) $ with $v_2^*(x_0)=0$ (where $x_0$ is some fixed state), satisfying
	\begin{equation}\label{main equation2}
		e^{v(x)+\rho}= {\inf_{ \varphi \in \mathcal{P}(A)} \int_B\int_A e^{c_1(x,a,b)}\int_X e^{v(y)}P(dy|x,a,b) \varphi (da) \Psi(x) (db)}.
	\end{equation} In addition, a strategy $\Phi^*\in S_1$ is an optimal strategy of player 1 given player 2 chooses $\Psi$ if and only if \eqref{main equation2} attains point-wise minimum at $\Phi^*$. Moreover, 
	\begin{equation}\label{rho representation2}
		\rho^*_1=\inf_{\pi^1\in \Pi_1}\limsup_{n\to \infty}\frac{1}{n}lnE^{\pi^1,\Psi}_x \big[e^{\sum_{t=0}^{n-1}c_1(x_t,a_t,b_t)}\big].
	\end{equation} 
	\hspace{3cm}$\bigg(:=\rho^{*\Psi}_1=\inf_{\pi^1\in \Pi_1}J^{\pi^1,\Psi}_1\bigg) $
\end{thm}

	\section{Existence of Nash equilibrium}

In this section we establish the existence of a pair of stationary equilibrium strategies for a nonzero-sum game. To this end we first outline a standard procedure for establishing the existence
of a Nash equilibrium. From Theorem \ref{theoremmain2} it follows that given that player 2 is using the strategy $\Psi \in S_2$, we can find  an optimal response $\Phi^* \in S_1$ for player 1 . Clearly $\Phi^*$ depends on $\Psi$ and moreover there may be several optimal responses for player 1 in $S_1$. Analogous results holds for player 2 if player 1 announces that he is going to use a strategy $\Phi \in S_1$. Hence given a pair of strategies $\left(\Phi, \Psi\right) \in S_1 \times S_2$, we can find a set of pairs of optimal responses $\left\{\left(\Phi^*, \Psi^*\right) \in S_1 \times S_2\right\}$ via the appropriate pair of optimality equations described above. This defines a set-valued map. Clearly any fixed point of this set-valued map is a Nash equilibrium.

To ensure the existence of a Nash equilibrium, we first take the following separability assumptions.

\begin{assumption}\label{assumption1.11}
	\item[(i)]	There exist two sub stochastic kernels
	\begin{equation*}
		P_1: X \times A \rightarrow \mathcal{P}(X), \quad P_2: X \times B \rightarrow \mathcal{P}(X)
	\end{equation*}
	such that
	\begin{equation*}
		P(\cdot \mid x, a, b)=P_1(\cdot \mid x,a)+P_2(\cdot \mid x, b), \quad x \in X, \quad a \in A, \quad b \in B .
	\end{equation*} Since $P << \lambda$, we have $P_1 << \lambda$ and $P_2 << \lambda$. Let  $h_1$ and $h_2$ be the respective densities. We assume that for each $x,y \in X$, $h_1(x,\cdot,y)$ and $h_2(x,\cdot,y)$ are continuous.
	
	\item[(ii)]\label{assumption4.1.ii} For each $x\in X$,
	\begin{equation*} \int_X \bigg[\sup_a|h_1(x,a,y)|+\sup_b|h_1(x,b,y)|\bigg]\lambda(dy)<\infty\end{equation*}
\end{assumption}
\begin{assumption}\label{assumption1.12}
	The reward functions $c_i, i=1,2$, are separable in action variables, i.e., there exist bounded continuous(in the second variable) functions
	\begin{equation*}
		c_{i 1}: X \times A \rightarrow \mathbb{R}, \quad c_{i 2}: X \times B \rightarrow \mathbb{R}, \quad i=1,2,
	\end{equation*}
	such that
	\begin{equation*}
		c_i(x, a, b)=c_{i 1}(x, a)+c_{i 2}(x, b), \quad x \in X, \quad a \in A, \quad b \in B .
	\end{equation*}
\end{assumption}

Following \cite{Himmelberg1976existence} and \cite{Parthasarathy1982existence} we topologize the spaces $S_1$ and $S_2$ with the topology of relaxed controls introduced in \cite{Warga1967functions}. We identify two elements $\Phi, \hat{\Phi} \in S_1$ if $\Phi=\hat{\Phi}$ a.e. $ \lambda$ (where $ \lambda$ is as in Assumption \ref{assumption1.4}). Let

$Y_1=\{f: X \times A \rightarrow \mathbb{R} \mid f$ is measurable in the first argument and continuous in the second and there exists $g \in L^1( \lambda)$ such that $|f(x, a)| \leq g(x)$ for every $a \in A\}$.

Then $Y_1$ is a Banach space with norm \cite{Warga1967functions}
\begin{equation*}
	\|f\|_W=\int_X \sup _a|f(x, a)|  \lambda(d x) .
\end{equation*}

Every $\Phi \in S_1$ (with the $ \lambda$-a.e. equivalence relation) can be identified with the element $\Lambda_{\Phi} \in Y_1^*$ (the dual of $Y_1$) defined as
\begin{equation*}
	\Lambda_{\Phi}(f)=\int_X \int_A f(x, a) \Phi(x)(d a)  \lambda(d x) .
\end{equation*}

Thus $S_1$ can be identified with a subset of $Y_1^*$. Equip $S_1$ with the weak-star topology. Then it can be shown as in \cite{Parthasarathy1982existence} that $S_1$ is compact and metrizable. $S_2$ can be topologized analogously.

Next, we present the following lemmas, which play a pivotal role to show  upper semi-continuity of a specific set-valued map which we have mentioned earlier.

\begin{lem}\label{convergence of r}
	Let, $\Phi_m \to {\Phi}\in S_1 ~\text{and}~ \Psi_m \to \Psi\in S_2$ in the weak star topology. Then under Assumption \ref{assumption1.12} for $i=1,2$, $\hat{c}_i(x,\Phi_{m},\Psi_{m}) \to \hat{c}_i(x,{\Phi},{\Psi})$ as $m\to\infty$.
\end{lem}

\begin{proof} We have for $i=1,2$,
	\begin{equation*}
		\begin{aligned}
			&\hat{c}_i(x,\Phi_m,\Psi_m)=ln\tilde{c}_i(x,\Phi_m,\Psi_m)\\
			&=ln\int_B\int_A e^{c_i(x,a,b)}\Phi_{m}(x) (da) \Psi_{m}(x) (db)\\
			&=ln\int_A e^{c_{i1}(x,a)}\Phi_{m}(x) (da)+ln\int_Be^{c_{i2}(x,b)} \Psi_{m}(x) (db)\\
			&=ln\int_X 1. \lambda(dx)\int_A e^{c_{i1}(x,a)}\Phi_{m}(x) (da)+ln\int_X 1. \lambda(dx)\int_Be^{c_{i2}(x,b)} \Psi_{m}(x) (db)\\
			&=ln\int_X \int_A \bigg(1.e^{c_{i1}(x,a)}\bigg)\Phi_{m}(x) (da) \lambda(dx)+ln \int_X\int_B\bigg(1.e^{c_{i2}(x,b)}\bigg) \Psi_{m}(x) (db) \lambda(dx).
		\end{aligned}
	\end{equation*}
	
	Now by  Assumption \ref{assumption1.12} and since  $\Phi_{m} \to \Phi ~\text{and}~ \Psi_{m} \to \Psi$ in the weak star topology, the result is immediate.
\end{proof}
\begin{lem}\label{covergence}
	Suppose Assumptions \ref{assumption1}, \ref{assumption1.4}, \ref{assumption1.11} and \ref{assumption1.12} hold. Let $\{v_m\}$ be a uniformly bounded sequence in $\mathbb{B}(X)$ and  $v\in \mathbb{B}(X)$ be a weak star limit point of $\{v_m\}$. If $\Phi_{m} \to {\Phi}\in S_1 ~\text{and}~ \Psi_m \to \Psi\in S_2$ in the weak star topology, then for each $x\in X$ and $i=1,2$,
	\begin{equation*}
		\int_{X}{v_m(y)}\hat{P_i}(dy|x,\Phi_{m},\Psi_m)\to \int_{X}{v(y)}\hat{P_i}(dy|x,\Phi,\Psi)~~~~\text{as}~~~~ m \rightarrow \infty.
	\end{equation*}
\end{lem}
\begin{proof} Note that 
	\begin{equation}\label{line364}
		\begin{aligned}
			&\Bigg|\int_{X}{v_m(y)}\hat{P_i}(dy|x,\Phi_{m}(x),\Psi_{m}(x))-\int_{X}{v(y)}\hat{P_i}(dy|x,\Phi(x),\Psi(x))\Bigg|\\
			&\leq\Bigg|\int_{X}{v_m(y)}\hat{P_i}(dy|x,\Phi_{m}(x),\Psi_{m}(x))-\int_{X}{v(y)}\hat{P_i}(dy|x,\Phi_{m}(x),\Psi_{m}(x))\Bigg|\\
			&+\Bigg|\int_{X}{v(y)}\hat{P_i}(dy|x,\Phi_{m}(x),\Psi_{m}(x))-\int_{X}{v(y)}\hat{P_i}(dy|x,\Phi(x),\Psi(x))\Bigg|.\\
		\end{aligned}
	\end{equation}
	We claim that 
	\begin{equation}\label{364'}
		\int_{X}{v(y)}\hat{P_i}(dy|x,\Phi_{m},\Psi_m)\to \int_{X}{v(y)}\hat{P_i}(dy|x,\Phi,\Psi)~~~~\text{as}~~~~ m \rightarrow \infty.
	\end{equation}
	Observe that, under Assumption \ref{assumption1.11}$(i)$, Assumption \ref{assumption1.12} and using \eqref{P2} we get 
	\begin{equation*}
		\begin{aligned}
			&\int_{X}{v(y)}\hat{P_i}(dy|x,\Phi_{m},\Psi_m)\\
			&=\frac{\int_{B}\int_A\int_{X}v(y)e^{c_{i1}(x,a)}e^{c_{i2}(x,b)}P_1(dy|x,a)\Phi_{m}(x)(da)\Psi_{m}(x)(db)}{c_i(x,\Phi_{m}(x),\Psi_{m}(x))}+\\
			&\hspace{3cm}\frac{\int_{B}\int_A\int_{X}v(y)e^{c_{i1}(x,a)}e^{c_{i2}(x,b)}P_2(dy|x,a)\Phi_{m}(x)(da)\Psi_{m}(x)(db)}{c_i(x,\Phi_{m}(x),\Psi_{m}(x))}\\
			&=\frac{\int_A\int_{X}v(y)e^{c_{i1}(x,a)}h_1(x,a,y)\lambda(dy)\Phi_{m}(x)(da)\int_B e^{c_{i2}(x,b)}\Psi_{m}(x)(db)}{c_i(x,\Phi_{m}(x),\Psi_{m}(x))}+\\
			&\hspace{2.5cm}\frac{\int_B\int_{X}v(y)e^{c_{i2}(x,b)}h_2(x,b,y)\lambda(dy)\Psi_{m}(x)(db)\int_A e^{c_{i1}(x,a)}\Phi_{m}(x)(da)}{c_i(x,\Phi_{m}(x),\Psi_{m}(x))}.
		\end{aligned}	
	\end{equation*}
	
	By using Lemma \ref{convergence of r} and since $\Phi_{m} \to {\Phi}\in S_1 ~\text{and}~ \Psi_m \to \Psi\in S_2$, under  Assumption \ref{assumption4.1.ii}$(i)$, \eqref{364'} holds true, i.e. the second term in the right hand side of \eqref{line364} goes to zero as $m \to \infty$. Now we show that the first one also goes to zero.
	
	Again note that,
	\begin{equation*}
		\begin{aligned}
			&\Bigg|\int_{X}{v_m(y)}\hat{P_i}(dy|x,\Phi_{m}(x),\Psi_{m}(x))-\int_{X}{v(y)}\hat{P_i}(dy|x,\Phi_{m}(x),\Psi_{m}(x))\Bigg|\\
			&=\Bigg|\int_{X}\int_B\int_{A}\frac{{v_m(y)}-{v(y)}}{\tilde{c}_i(x,\Phi_{m}(x),\Psi_{m}(x))}e^{c_i(x,a,b)}P(dy|x,a,b)\Phi_{m}(x)(da)\Psi_{m}(x)(db)\Bigg|\\
			&\leq e^{\bar{c}} \int_B\int_{A}\Bigg|\int_{X}{({v_m(y)}-{v(y)})}P(dy|x,a,b)\Bigg|\Phi_{m}(x)(da)\Psi_{m}(x)(db)\\
			&\leq e^{\bar{c}}~\sup_{b\in B}\sup_{a\in A}\Bigg|\int_{X}{({v_m(y)}-{v(y)})}P(dy|x,a,b)\Bigg|
			\\
			&=e^{\bar{c}}~ \sup_{b\in B}\sup_{a\in A}\Bigg|\int_{X}{({v_m(y)}-{v(y)})}h(x,a,b,y) \lambda(d y)\Bigg|,
		\end{aligned}
	\end{equation*} where $h=h_1+h_2$.
	From the compactness of $A, B$ and the continuity of $h(x,.,.,y)$, it follows that for  $ m\in \mathbb{N},$
		{\small\begin{equation*}
			\begin{aligned}
				V_m(x)&:=\Bigg|\int_{X}{({v_m(y)}-{v(y)})}h(x,a_m,b_m,y) \lambda(d y)\Bigg|=\sup_{b\in B}\sup_{a\in A}\Bigg|\int_{X}{({v_m(y)}-{v(y)})}h(x,a,b,y) \lambda(d y)\Bigg| ,
			\end{aligned}
	\end{equation*}} for some sequences $\{a_m\}\in A$ and $\{b_m\}\in B$. We now prove that $V_m(x)\to 0$ as $m\to \infty$. 
	
	Since $A$ and $B$ are compact , without loss of generality, we can assume that
	\begin{equation*}
		a_{m} \rightarrow a_0 \quad \text { and } \quad b_{m} \rightarrow b_0, \quad \text { for some } a_0 \in A \text { and } b_0 \in B.
	\end{equation*}
	Note that, for each $m$, we have
	\begin{equation}\label{1.17}
		\begin{aligned}
			V_{m}(x)  &\leq\left|\int\left({v_m(y)}-{v(y)}\right)\left(h\left(x, a_{m}, b_{m},y\right)-h\left(x, a_0, b_0,y\right)\right)  \lambda(dy)\right|+\\
			&\hspace{5.5cm}
			\left|\int\left({v_m(y)}-{v(y)}\right) h\left(x, a_0, b_0,y\right)  \lambda(dy)\right|.
		\end{aligned}
	\end{equation}
	Moreover,
	\begin{equation*}
		\begin{aligned}
			&\left|\int\left({v_m(y)}-{v(y)}\right)\left(h\left(x, a_{m}, b_{m},y\right)-h\left(x, a_0, b_0,y\right)\right)  \lambda(dy)\right|
			\leq \\
			&\hspace{4.5cm}\left\|{v_m(y)}-{v(y)}\right\|\left\|h\left(x,  a_{m}, b_{m},\cdot\right)-h\left(x,  a_0, b_0,\cdot\right)\right\|_{L^{1}( \lambda)}.
		\end{aligned}
	\end{equation*}
	
	By Assumption \ref{assumption1.11}$(ii)$
	and by the boundedness of $\left\{\left\|{v_m(y)}-{v(y)}\right\|\right\}$, from the last inequality we get that, first term on the right-hand side of \eqref{1.17} goes to zero as $m \rightarrow \infty$. Since, $v$ is a weak star limit of $\{v_{m}\}$ and $h\left(x, a_0, b_0,\cdot\right) \in L_1( \lambda)$, so the second term on the right-hand side of \eqref{1.17} also goes to zero as $m \rightarrow \infty$. Thus, we have shown that $V_{m}(x) \rightarrow 0$ as $m \rightarrow \infty$. Hence  the result followed. \end{proof}

\begin{lem}\label{boundedness}
	Let for $M>0$, $\{v_m\}$ be any sequence in $\mathbb{B}_M(X)$ with $v_m(x_0)=0$ for all $m\in \mathbb{N}$. Then  $\{{v_m}\}$ is  uniformly bounded .
\end{lem}

\begin{proof} Now for each fixed $m\in \mathbb{N}$ as, 
	\begin{equation*}\inf_{x\in X} v_m(x)\leq v_m(x_0)=0\end{equation*}
	and $\|v_m(x)\|_{sp}\leq M$, using \eqref{span norm}  we have
	\begin{equation}\label{1}
		\sup_{x\in X} v_m{(x)}\leq M
	\end{equation}
	Again as $\sup_{x\in X} v_m(x)\geq v_m(x_0)$ and $\|v_m(x)\|_{sp}\leq M$, using \eqref{span norm} we have
	\begin{equation}\label{2}
		\inf_{x\in X} v_m{(x)}\geq -M
	\end{equation}	
	Now from, \eqref{1} and \eqref{2} for fixed $m\in \mathbb{N}$ and each $x\in X$ we get
	\begin{equation*}
		-M\leq v_m(x)\leq M
	\end{equation*}
	Therefore, $\{{v_m}\}$ is a uniformly bounded sequence in $\mathbb{B}_M(X)$.
\end{proof}

For a fix $\Phi \in S_1$ let, 
\begin{equation*}
	\begin{aligned}
		&H(\Phi)=\bigg\{\Psi^* \in S_2:\hat{c}_2(x,\Phi(x),\Psi^*(x))+ln \int_X e^{v^{*\Phi}_2(y)}\hat{P}_2(dy|x,\Phi(x),\Psi^*(x))\\ & \hspace{3.5cm}=\inf_{\psi \in \mathcal{P}(B)}\bigg[\hat{c}_2(x,\Phi(x),\psi)+ln \int_X e^{v^{*\Phi}_2(y)}\hat{P}_2(dy|x,\Phi(x),\psi)\bigg]\bigg\},
	\end{aligned}
\end{equation*}
where ${v^{*\Phi}_2}$ is the unique solution of \eqref{main equation1} corresponding to the strategy  $\Phi \in S_1$.

Similarly  	for a fix $\Psi \in S_2$ let, 
\begin{equation*}
	\begin{aligned}
		&H(\Psi)=\bigg\{\Phi^* \in S_1:\hat{c}_1(x,\Phi^*(x),\Psi(x))+ln \int_X e^{v^{*\Psi}_1(y)}\hat{P}_1(dy|x,\Phi^*(x),\Psi(x))\\ & \hspace{3.5cm}=\inf_{\varphi \in \mathcal{P}(A)}\bigg[\hat{c}_1(x,\varphi,\Psi(x))+ln \int_X e^{v^{*\Psi}_1(y)}\hat{P}_2(dy|x,\varphi,\Psi(x))\bigg]\bigg\},
	\end{aligned}
\end{equation*}where ${v^{*\Psi}_1}$ is the unique solution of \eqref{main equation2} corresponding to the strategy  $\Psi \in S_2$. 
\begin{rem}\label{remark2}
	Since exponential and logarithmic functions are increasing functions, so $H(\Phi)$ and $H(\Psi)$ also  have the following expressions:
	\begin{equation*}
		\begin{aligned}
			&H(\Phi)=\bigg\{\Psi^* \in S_2: \int_B\int_A e^{c_2(x,a,b)}\int_X e^{v^{*\Phi}_2(y)}P(dy|x,a,b)\Phi(x) (da)  \Psi^*(x) (db)\\ &\hspace{3cm}= \inf_{ \psi \in \mathcal{P}(B)} \int_B\int_A e^{c_2(x,a,b)}\int_X e^{v^{*\Phi}_2(y)}P(dy|x,a,b)\Phi(x) (da)  \psi (db)\bigg\}.
		\end{aligned}
	\end{equation*}
	
	\begin{equation*}
		\begin{aligned}
			&H(\Psi)=\bigg\{\Phi^* \in S_1: \int_B\int_A e^{c_1(x,a,b)}\int_X e^{v^{*\Psi}_1(y)}P(dy|x,a,b)\Phi^*(x) (da)  \Psi(x) (db)\\ &\hspace{3cm}= \inf_{ \varphi \in \mathcal{P}(A)} \int_B\int_A e^{c_1(x,a,b)}\int_X e^{v^{*\Psi}_1(y)}P(dy|x,a,b)\varphi (da)  \Psi(x) (db)\bigg\}.
		\end{aligned}
	\end{equation*}
\end{rem} 
\vspace{0.5cm}
Next set 
\begin{equation*}
	H(\Phi,\Psi)=H(\Psi)\times H(\Phi).
\end{equation*}

\begin{lem}\label{nonempty closed convex }
	Under  Assumptions \ref{assumption1}, \ref{assumption1.4},  \ref{assumption1.11} and \ref{assumption1.12}, 	$H(\Phi,\Psi)$ is a non-empty compact convex subset of $S_1\times S_2$. 
\end{lem}
\begin{proof} 
	From Remark \ref{remark1}, we know that $\int_X e^{v(y)}\hat{P}_2(dy|x,\Phi(x),\psi)$ is continuous on $\mathcal{P}(A)\times\mathcal{P}(B)$ for each $x\in X$. As $B$ is compact, $\mathcal{P}(B)$ is also compact. Then it is easy to see that $H(\Phi)$ is  non-empty.
	Let  $\Psi_{m}^* \in H(\Phi)$ and as $S_2$ is compact, $\{\Psi^*_{m}\}$ has a convergent subsequence(denoted by the same sequence by abuse of notation)such that $\Psi^*_{m} \to \hat{\Psi}\in S_2$. Now for any $\psi\in \mathcal{P}(B)$
	\begin{equation}\label{closed equation 1}
		\begin{aligned}
			&\hat{c}_2(x,\Phi(x),\Psi_{m}^*(x))+ln \int_X e^{{v^{*\Phi}_2}(y)}\hat{P}_2(dy|x,\Phi(x),\Psi_{m}^*(x))\\&\hspace{5cm}\leq \hat{c}_2(x,\Phi(x),\psi)+ln \int_X e^{{v^{*\Phi}_2}(y)}\hat{P}_2(dy|x,\Phi(x),\psi).
		\end{aligned}
	\end{equation}
	Using Lemma \ref{convergence of r} and \ref{covergence}, from \eqref{closed equation 1} we get for any $\psi \in \mathcal{P}(B)$
	\begin{equation*}
		\begin{aligned}
			&\hat{c}_2(x,\Phi(x),\hat{\Psi}(x))+ln \int_X e^{{v^{*\Phi}_2}(y)}\hat{P}_2(dy|x,\Phi(x),\hat{\Psi}(x))\\&\hspace{4cm}\leq \hat{c}_2(x,\Phi(x),\psi)+ln \int_X e^{{v^{*\Phi}_2}(y)}\hat{P}_2(dy|x,\Phi(x),\psi)\hspace{0.25cm} \lambda-a.e.
		\end{aligned}
	\end{equation*}Hence it follows that $\hat{\Psi}\in H(\Phi)$ and therefore  $H(\Phi)$ is closed. Since $S_2$ is a compact metric space, it follows that $H(\Phi)$ is also compact. Using  Remark \ref{remark2} the convexity of  $H(\Phi)$ and  $H(\Psi)$ follows easily. By analogous arguments, $H(\Psi)$ is also non-empty compact subset of $S_2$.
	Hence  $H(\Phi,\Psi)$ is a non-empty compact convex  subset of $S_1\times S_2$.
\end{proof}

Next lemma proves the upper semi-continuity of a certain set valued map. This result will
be useful in establishing the existence of a Nash equilibrium in the space of stationary
Markov strategies.

\begin{lem}\label{uppersemicont.}
	Under  Assumptions \ref{assumption1}, \ref{assumption2}, \ref{assumption1.4}, \ref{assumption1.5},   \ref{assumption1.11} and \ref{assumption1.12}  the map $(\Phi,\Psi)\to H(\Phi,\Psi)$ from $S_1\times S_2 \to 2^{S_1\times S_2}$ is upper semi-continuous.
\end{lem}

\begin{proof} 	Let $\Psi_{m}^* \in H(\Phi_{m})$. $\{\Phi_{m}\}$ has a convergent subsequence (denoted by the same sequence by abuse of notation) such that $\Phi_{m} \to \bar{\Phi}\in S_1$ and similarly  $\{\Psi_{m}^*\}$ has a subsequence too such that $\Psi_{m}^* \to \hat{\Psi}\in S_2$. Since, $\{v^{*\Phi_{m}}_2\}\in  \mathbb{B}_L(X)$ and $\{\rho^{*\Phi_{m}}_2\}$ is bounded so without loss of generality let $v^{*\Phi_{m}}_2\to v_2$ in the weak star sense and $\rho^{*\Phi_{m}}_2 \to \rho_2 $. Then since,
	\begin{equation}\label{key1}
		{\rho^{*\Phi_{m}}_2}+{v^{*\Phi_{m}}_2}=\hat{c}_2(x,\Phi_{m}(x),\Psi_{m}^*(x))+ln \int_X e^{{v^{*\Phi_{m}}_2}(y)}\hat{P}_2(dy|x,\Phi_{m}(x),\Psi_{m}^*(x))
	\end{equation}
	Using Lemma \ref{convergence of r}, \ref{covergence} and \ref{boundedness} it follows that
	\begin{equation}\label{keyA}
		\rho_2+v_2(x)= \hat{c}_2(x,\bar{\Phi}(x),\hat{\Psi}(x))+ln \int_X e^{{v_2(y)}}\hat{P}_2(dy|x,\bar{\Phi}(x),\hat{\Psi}(x))\hspace{0.5cm} \lambda-a.e.
	\end{equation}
	From \eqref{key1} for any $\psi \in \mathcal{P}(B)$ we get 
	\begin{equation*}\label{key2}
		{\rho^{*\Phi_{m}}_2}+{v^{*\Phi_{m}}_2}\leq \hat{c}_2(x,\Phi_{m}(x),\psi)+ln \int_X e^{{v^{*\Phi_{m}}_2}(y)}\hat{P}_2(dy|x,\Phi_{m}(x),\psi).
	\end{equation*}
	Again using Lemma \ref{convergence of r}, \ref{covergence} and  \ref{boundedness} it follows that
	\begin{equation}\label{key3}
		\rho_2+v_2(x)\leq \hat{c}_2(x,\bar{\Phi}(x),\psi)+ln \int_X e^{{v_2(y)}}\hat{P}_2(dy|x,\bar{\Phi}(x),\psi)\hspace{0.5cm} \lambda-a.e.
	\end{equation}
	Let $v^*_2(x)=v_2(x)-v_2(x_0)$. Then from \eqref{key3} we get, for any $\psi \in \mathcal{P}(B)$
	\begin{equation}\label{key4}
		\rho_2+v^*_2(x)\leq \hat{c}_2(x,\bar{\Phi}(x),\psi)+ln \int_X e^{{v^*_2(y)}}\hat{P}_2(dy|x,\bar{\Phi}(x),\psi)\hspace{0.5cm} \lambda-a.e.
	\end{equation} and from \eqref{keyA} we get 
	\begin{equation}\label{key5}
		\begin{aligned}
			\rho_2+v^*_2(x)&= \hat{c}_2(x,\bar{\Phi}(x),\hat{\Psi}(x))+ln \int_X e^{{v^*_2(y)}}\hat{P}_2(dy|x,\bar{\Phi}(x),\hat{\Psi}(x))\hspace{0.5cm} \\
			&\geq \inf_{\psi \in \mathcal{P}(B)}\bigg[\hat{c}_2(x,\bar{\Phi}(x),\psi)+ln \int_X e^{{v^*_2(y)}}\hat{P}_2(dy|x,\bar{\Phi}(x),\psi)\bigg]\hspace{0.5cm}
		\end{aligned} \Biggl\}\lambda-a.e.
	\end{equation}
	Since \eqref{key4} holds for every $\psi \in \mathcal{P}(B)$, from \eqref{key4} and \eqref{key5} we get 
	\begin{equation}\label{key7}
		\rho_2+v^*_2(x)=\inf_{\psi \in \mathcal{P}(B)}\bigg[\hat{c}_2(x,\bar{\Phi}(x),\psi)+ln \int_X e^{{v^*_2(y)}}\hat{P}_2(dy|x,\bar{\Phi}(x),\psi)\bigg]\hspace{0.5cm} \lambda-a.e. 
	\end{equation} with $v^*_2(x_0)=0$.
	Now by Theorem \ref{theoremmain1} we can say that \eqref{key7} has unique solution $(\rho^{*\bar{\Phi}}_2,v^{*\bar{\Phi}}_2)$(corresponds to $\bar{\Phi}\in S_1$) satisfying $v^{*\bar{\Phi}}_2(x_0)=0$. Therefore, $\rho_2=\rho^{*\bar{\Phi}}_2$ and $v^*_2=v^{*\bar{\Phi}}_2$. Thus, from \eqref{key5} and \eqref{key7} it follows that $\hat{\Psi}\in H(\bar{\Phi})$. 
	
	Suppose $\Phi_{m}^* \in H(\Psi_{m})$ and along a suitable subsequence $\Phi_{m}^* \to \hat{\Phi}\in S_1$ and $\Psi_{m} \to \bar{\Psi}\in S_2$. Then by similar arguments one can show  that $\hat{\Phi}\in H(\bar{\Psi})$. This proves that $(\hat{\Phi},\hat{\Psi})\in H(\bar{\Phi},\bar{\Psi})$. Hence the map $(\Phi,\Psi)\to H(\Phi_1,\Psi)$  is upper semi-continuous.
	
	Now we are now all set to show the existence is of the Nash equilibrium which is directly follows by using Fan's fixed point theorem \cite{Fan1952fixed}.
\end{proof}

\begin{thm}\label{Nash equilibrium}
	Suppose that the  Assumptions \ref{assumption1}, \ref{assumption2}, \ref{assumption1.4}, \ref{assumption1.5},  \ref{assumption1.11} and \ref{assumption1.12} are satisfied. Then   there exists a Nash equilibrium in the space of stationary strategies $S_1\times S_2$.
\end{thm}

\begin{proof} Using Lemma \ref{nonempty closed convex } and  \ref{uppersemicont.} from Fan's fixed point theorem, it follows that there exists a fixed point $(\Phi^*,\Psi^*)\in (S_1\times S_2)$ for the map $(\Phi,\Psi)\to H(\Phi,\Psi)$ from $S_1\times S_2 \to 2^{S_1\times S_2}$. 	This implies that $(\rho_1^{*\Psi^*}, v_1^{*\Psi^*}),(\rho_2^{*\Phi^*}, v_2^{*\Phi^*})$ satisfy the following coupled optimality equations:

\begin{equation}\label{optimality1}
	\begin{aligned}
		\rho_1^{*\Psi^*}+ v_1^{*\Psi^*}&=\inf_{\varphi \in \mathcal{P}(A)}\bigg[\hat{c}_1(x,\varphi,\Psi^*(x))+ln \int_X e^{v_1^{*\Psi^*}}(y)\hat{P}_1(dy|x,\varphi,\Psi^*(x))\bigg] \\
		&=\bigg[\hat{c}_1(x,\Phi^*(x),\Psi^*(x))+ln \int_X e^{v_1^{*\Psi^*}}(y)\hat{P}_1(dy|x,\Phi^*(x),\Psi^*(x))\bigg],
	\end{aligned}
\end{equation}

and

\begin{equation}\label{optimality2}
	\begin{aligned}
		\rho_2^{*\Phi^*}+ v_2^{*\Phi^*}&=\inf_{\psi \in \mathcal{P}(B)}\bigg[\hat{c}_2(x,\Phi^*(x),\psi)+ln \int_X e^{v_2^{*\Phi^*_{1}}}\hat{P}_2(dy|x,\Phi^*(x),\psi)\bigg] \\
		&=\bigg[\hat{c}_2(x,\Phi^*(x),\Psi^*(x))+ln \int_X e^{v_2^{*\Phi^*_{1}(x)}}\hat{P}_1(dy|x,\Phi^*(x),\Psi^*(x))\bigg].
	\end{aligned} 
\end{equation}

Now by Theorem \ref{theoremmain1}, from \eqref{optimality1}, it follows that
\begin{equation}\label{Nash1}
	\rho_1^{*\Psi^*}=\inf_{\pi^1\in \Pi_1}J^{\pi^1,\Psi^*}_1 =J^{\Phi^*,\Psi^*}_1 .
\end{equation}

Similarly, by Theorem \ref{theoremmain2}, from \eqref{optimality2}, it follows that
\begin{equation}\label{Nash2}
	\rho^{*\Phi^*}_2=\inf_{\pi^2\in \Pi_2}J^{\Phi^*,\pi^2}_2=J^{\Phi^*,\Psi^*}_2.
\end{equation}

Thus, from equations \eqref{Nash1} and \eqref{Nash2}, we get
\begin{equation}\label{Nash}
	\begin{aligned}
		& J^{\pi_1,\Psi^*}_1 \geq J^{\Phi^*,\Psi^*}_1, \forall \pi^1 \in \Pi_1, \\
		& J^{\Phi^*,\pi^2}_2\geq J^{\Phi^*,\Psi^*}_2, \forall \pi^2 \in \Pi_2 .
	\end{aligned}\Biggl\}
\end{equation}

Hence $(\Phi^*,\Psi^*) \in S_1 \times S_2$ obviously forms a $ \lambda$-equilibrium stationary strategy (i.e., \eqref{Nash} holds in a set of $ \lambda$-measure 1). Then by a construction analogous to Theorem 1 of \cite{Parthasarathy1982existence} the existence of the desired Nash equilibrium
follows.
\end{proof}

\bibliographystyle{plain}
\bibliography{bib}

\end{document}